\input ./style/arxiv-vmsta.cfg
\documentclass[numbers,compress,v1.0.1]{vmsta}

\usepackage{mathbh}

\newtheorem{thm}{Theorem}
\newtheorem{cor}{Corollary}
\newtheorem{lemma}{Lemma}

\theoremstyle{definition}

\newtheorem{ex}{Example}
\def\ind{\mathbh{1}}

\volume{4}
\issue{1}
\pubyear{2017}
\firstpage{65}
\lastpage{78}
\doi{10.15559/17-VMSTA74}

\startlocaldefs

\allowdisplaybreaks
\endlocaldefs

\begin{document}
\begin{frontmatter}

\title{Randomly stopped maximum and maximum of sums with consistently varying distributions}

\author{\inits{I.M.}\fnm{Ieva Marija}\snm{Andrulyt\.{e}}}\email{i.m.andrulyte@gmail.com}
\author{\inits{M.}\fnm{Martynas}\snm{Manstavi\v{c}ius}}
\email{martynas.manstavicius@mif.vu.lt}
\author{\inits{J.}\fnm{Jonas}\snm{\v Siaulys}\corref{cor1}}\email{jonas.siaulys@mif.vu.lt}
\cortext[cor1]{Corresponding author.}

\address{Faculty of Mathematics and Informatics, Vilnius University, Naugarduko 24, Vilnius~LT-03225, Lithuania}

\markboth{I.M. Andrulyt\.{e} et al.}{Randomly stopped maximum and maximum of sums with consistently varying distributions}

\begin{abstract}
Let $\{\xi_1,\xi_2,\ldots\}$ be a sequence of independent random
variables, and~$\eta$ be a counting random variable independent of this
sequence. In addition, let $S_0:=0$ and $S_n:=\xi_1+\xi_2+\cdots+\xi_n$
for $n\geqslant1$. We consider conditions for random variables
$\{\xi_1,\xi_2,\ldots\}$ and $\eta$ under which the distribution
functions of the random maximum
$\xi_{(\eta)}:=\max\{0,\xi_1,\xi_2,\ldots,\xi_\eta\} $ and of the
random maximum of sums $S_{(\eta)}:=\max\{S_0,S_1,S_2,\ldots,S_\eta\}$
belong to the class of consistently varying distributions. In our
consideration the random variables $\{\xi_1,\xi_2,\ldots\}$ are not
necessarily identically distributed.
\end{abstract}


\received{10 December 2016}
%
\revised{23 January 2017}
%
\accepted{27 January 2017}
\publishedonline{6 March 2017}
\end{frontmatter}

\section{Introduction}\label{i}

Let $\{\xi_1,\xi_2,\ldots\}$ be a sequence of independent random
variables (r.v.s) with distribution functions (d.f.s) $\{F_{\xi
_1},F_{\xi_2},\ldots\}$, and let $\eta$ be a counting r.v., that~is,
an~integer-valued, nonnegative, and nondegenerate at zero r.v. In
addition, suppose that the r.v. $\eta$ and r.v.s $\{\xi_1,\xi_2,\ldots
\}$ are independent.

Let
$S_0=0$, $S_n=\xi_1+\xi_2+\cdots+\xi_n$ for $n\in\mathbb{N}$, and let
\begin{equation*}
S_\eta=\sum\limits
_{k=1}^{\eta}
\xi_k
\end{equation*}
be the randomly stopped sum of r.v.s $\{\xi_1,\xi_2,\ldots\}$.

Similarly, let
$\xi_{(n)}=\max\{0,\xi_1,\xi_2,\ldots,\xi_n\}$ for $n\in\mathbb{N}$,
$\xi_{(0)}=0$, and let
\begin{equation*}
\xi_{(\eta)}= %
\begin{cases}
\max\{0,\xi_1,\xi_2,\ldots,\xi_\eta\}& \textnormal{if}\ \eta\geqslant1,\\
0 &\textnormal{if}\  \eta=0
\end{cases} %
\end{equation*}
be the randomly stopped maximum of r.v.s $\{\xi_1,\xi_2,\ldots\}$.

Finally, let $S_{(n)}:=\max\{S_0,S_1,S_2,\ldots,S_n\}$ for $n\geqslant
0$, and let
\[
S_{(\eta)}:=\max\{S_0,S_1,S_2,
\ldots,S_\eta\}
\]
be the randomly stopped maximum of sums $S_0,S_1,S_2,\ldots$.

We denote the distribution functions (d.f.s) of $S_\eta$, $\xi_{(\eta
)}$, and $S_{(\eta)}$ by $F_{S_\eta}$, $F_{\xi_{(\eta)}}$, and
$F_{S_{(\eta)}}$ respectively, together with their tails $\overline
{F}_{S_\eta}$, $\overline{F}_{\xi_{(\eta)}}$, and $\overline
{F}_{S_{(\eta)}}$. For any positive $x$, we have
\begin{align*}
\overline{F}_{S_\eta}(x)&=\sum\limits_{n=1}^{\infty}\mathbb{P}(\eta =n)\mathbb{P}(S_n> x),\\
\overline{F}_{\xi_{(\eta)}}(x)&=\sum\limits_{n=1}^{\infty}\mathbb{P}(\eta=n)\mathbb{P}(\xi_{(n)}> x),\\
\overline{F}_{S_{(\eta)}}(x)&=\sum\limits_{n=1}^{\infty}\mathbb{P}(\eta =n)\mathbb{P}(S_{(n)}> x).
\end{align*}

In \cite{kss-2016}, conditions were found for the d.f.s $F_{S_\eta}$ to
belong to the class of consistently varying distributions. In this
paper, we are interested in sufficient conditions under which the d.f.s
$F_{\xi_{(\eta)}}$ and $F_{S_{(\eta)}}$ have consistently varying tails.

Throughout this paper, for two vanishing (at infinity) functions~$f$
and~$g$,\break $f(x) \mathop{=} o  ( g(x)  )$ means that $\lim_{x\to\infty} {f(x)}/{g(x)}=0$, and $f(x) \sim g(x)$
means that\break $\lim_{x\to\infty} {f(x)}/{g(x)}=1$. Also, we denote the support of a
counting r.v.\ $\eta$ by
\[
{\rm supp}(\eta):=\bigl\{n\in\mathbb{Z}_+:\mathbb{P}(\eta=n)>0\bigr\}.
\]

Before discussing the properties of $F_{\xi_{(\eta)}}$ and $F_{S_{(\eta
)}}$, we recall the definitions of some classes of heavy-tailed d.f.s.
For a~d.f. $F$, we denote $ \overline{F}(x) = 1-F(x)$ for real~$x$.

\begin{itemize}
\item \textit{A d.f. $F$ is heavy-tailed $(F \in
\mathcal{H })$
if for every $\delta> 0$,
\[
\lim\limits
_{x\rightarrow\infty} \overline{F}(x){\rm e}^{\delta x} = \infty.\vadjust{\eject}
\]
}
\item \textit{A d.f. $F$ is long-tailed $(F\in\mathcal{L})$
if for every $y$ $($equivalently, for some $y>0)$, ${\overline{F}(x+y)}\mathop{\sim}{\overline{F}(x)}$.}

\item \textit{A d.f. $F$ has dominatingly varying tail $(F\in\mathcal{D})$ if for every
$y\in(0,1)$  $($equi\-va\-lent\-ly, for some $y\in(0,1))$,
\begin{equation*}
\limsup_{x\rightarrow\infty}\frac{\overline{F}(\mathit{xy})}{\overline
{F}(x)}<\infty.
\end{equation*}
}
\item \textit{A d.f. $F$ has consistently varying tail $(F\in\mathcal{C})$ if
\begin{equation*}
\lim_{y\uparrow1}\limsup_{x\to\infty} \frac{\overline{F}(\mathit{xy})}{\overline
{F}(x)} =
1.
\end{equation*}
}
\item \textit{A d.f. $F$ has regularly varying tail $(F\in
\mathcal{R})$ if there is $\alpha\geqslant0$ such that
\begin{equation*}
\lim\limits
_{x\rightarrow\infty}\frac{\overline{F}(\mathit{xy})}{\overline
{F}(x)} = y^{-\alpha},
\end{equation*}
for all $y > 0$. }
\item \textit{A d.f. $F$ supported on the interval $[0,\infty)$ is subexponential $(F\in\mathcal{S})$ if
\begin{equation}
\label{tr}\lim\limits
_{x\rightarrow\infty}\frac{
\overline{F*F}(x)}{\overline{F}(x)} = 2.
\end{equation}
}%
\textit{If a d.f. $G$ is supported
on $\mathbb{R}$, then we say that $G$ is subexponential $(G\in\mathcal{S})$
if the d.f. $F(x)=G(x)\ind_{[0,\infty)}(x)$ satisfies relation \eqref{tr}.}
\end{itemize}

It is known (see, e.g., \cite{chist,eo-1984,k-1988}, and Chapters 1.4
and A3 in \cite{EKM}) that these classes satisfy the following inclusions:
\[
\mathcal{R} \subset\mathcal{C} \subset\mathcal{L} \cap\mathcal{D} \subset
\mathcal{S} \subset\mathcal{L} \subset\mathcal{H}, \quad \mathcal {D}\subset
\mathcal{H}.
\]
These inclusions are depicted in Fig.~\ref{f1} borrowed from the paper
\cite{kss-2016}. In this figure, the class $\mathcal{C}$ of
distributions having consistently varying tails is highlighted.
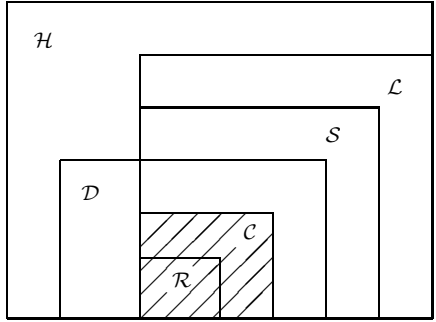
\begin{figure}[h]
\centering
\begin{picture}(150,120)
\put(40,0){\line(1,0){110}}
\put(40,0){\line(0,1){100}}
\put(150,0){\line(0,1){100}}
\put(40,100){\line(1,0){110}}
\put(-10,0){\line(1,0){50}}
\put(-10,0){\line(0,1){120}}
\put(-10,120){\line(1,0){160}}
\put(150,100){\line(0,1){20}}
\put(40,80){\line(1,0){90}}
\put(130,0){\line(0,1){80}}
\put(10,0){\line(0,1){60}}
\put(10,60){\line(1,0){100}}
\put(110,0){\line(0,1){60}}
\put(40,40){\line(1,0){50}}
\put(90,0){\line(0,1){40}}
\put(40,23){\line(1,0){30}}
\put(70,0){\line(0,1){23}}
\put(0,103){$\mathcal{H}$}
\put(133,85){$\mathcal{L}$}
\put(110,67){$\mathcal{S}$}
\put(18,45){$\mathcal{D}$}
\put(79,30){$\mathcal{C}$}
\put(52,12){$\mathcal{R}$}
\put(60,20){\line(1,1){20}}
\put(40,0){\line(1,1){10}}
\put(40,9){\line(1,1){31}}
\put(40,18){\line(1,1){22}}
\put(40,27){\line(1,1){13}}
\put(49,0){\line(1,1){27}}
\put(58,0){\line(1,1){32}}
\put(67,0){\line(1,1){23}}
\put(76,0){\line(1,1){14}}
\end{picture}
\caption{Classes of heavy-tailed distributions}
\label{f1}
\end{figure}

It should be noted that the subject of the paper is partially motivated
by the closure problem of the random convolution. In the case of
independent and identically distributed (i.i.d.) r.v.s $\{\xi,\xi_1,\xi
_2,\ldots\}$, we say that a class $\mathcal{K}$ of d.f.s is closed
with respect to the random convolution if the condition $F_\xi\in
\mathcal{K}$ implies $F_{S_\eta}\in\mathcal{K}$. The first result on
the convolution closure of subexponential distributions was obtained by
Embrechts and Goldie (see Thm.~4.2 in \cite{em+gol}) and by Cline (see
Thm.~2.13 in \cite{cline}).
\begin{thm}\label{tt1}
\textit{Let $\{\xi_1,\xi_2,\ldots\}$ be independent copies of a
nonnegative r.v. $\xi$ with subexponential d.f. $F_\xi$. Let $\eta$ be
a counting r.v. independent of $\{\xi_1,\xi_2,\ldots\}$. If $\mathbb
{E}(1+\delta)^\eta<\infty$
for some $\delta>0$, then the d.f. $F_{{S}_{\eta}}\in\mathcal{S}$.}
\end{thm}

The random closure results for class $\mathcal{D}$ can be found in \cite
{ds-2016,ls-2012}, and for the class $\mathcal{L}$, in \cite
{albin,ls-2012,wat,xu}. The random closure
results for the class $\mathcal{C}$ can be derived from the results of
\cite{kss-2016}. We note that in \cite{ds-2016,kss-2016,xu}, the case of not necessarily identically\vadjust{\eject} distributed r.v.s $\{
\xi_1,\xi_2,\ldots\}$ was considered. We further recall two known
results. First, in Theorem~\ref{tt2}, we give conditions for~$F_{S_\eta
}$ to belong to the class $\mathcal{D}$. Its proof can be found in \cite
[Thm.~2.1]{ds-2016}. Recall that a d.f. $F$ belongs to the class
$\mathcal{D}$ if and only if its upper Matuszewska index $J^{+}_{F} <
\infty$, where by definition
\begin{equation*}
J_F^+=-\lim\limits
_{y\rightarrow\infty}\frac{1}{\log y}\log \biggl(\liminf \limits
_{x\rightarrow\infty}
\frac{\overline{F}(xy)}{\overline
{F}(x)} \biggr).
\end{equation*}

\begin{thm}\label{tt2}
Let r.v.s $\{\xi_1,\xi_2,\ldots\}$ be nonnegative independent
but not necessary identically distributed, and $\eta$ be a counting
r.v. independent of $\{\xi_1,\xi_2,\ldots\}$. Then the d.f.\  $F_{S_\eta
}$ belongs to the class $\mathcal{D}$ if the following three conditions
are satisfied:
\begin{eqnarray*}
&{\rm(i)}& F_{\xi_\kappa}\in\mathcal{D}\ \ \mathit{for\ some}\ \kappa\in {\rm supp}(\eta),\\
&{\rm(ii)}& \limsup\limits_{x\rightarrow\infty} \sup\limits_{n\geqslant\kappa}\frac{1}{n\overline{F}_{\xi_\kappa}(x)}\sum\limits_{i=1}^{n}\overline{F}_{\xi_i}(x)<\infty,\\
&{\rm(iii)}& \mathbb{E}\eta^{p+1}<\infty\ \ \mathit{for\ some}\ p>J_{F_{\xi_\kappa}}^+.
\end{eqnarray*}
\end{thm}

A similar result for r.v.s having d.f.s with consistently varying tails
can be found in \cite[Thm.~6]{kss-2016}.

\begin{thm}\label{tt3}
Let $\{\xi_1,\xi_2,\ldots\}$ be independent real-valued r.v.s,
and $\eta$ be a coun\-ting r.v. independent of $\{\xi_1,\xi_2,\ldots\}
$. Then the d.f. $F_{S_\eta}$ belongs to the class $\mathcal{C}$ if the
following conditions are satisfied:
\begin{eqnarray*}
&{\rm(a)}& F_{\xi_{1}} \in\mathcal{C},\\
&{\rm(b)}& \mathit{for \ each}\ k \geqslant2, \ \mathit{either} \ F_{\xi_{k}} \in\mathcal{C} \ \mathit{or} \ \overline{F}_{\xi_{k}}(x) \mathop{=} o \bigl(\overline{F}_{\xi_{1}}(x) \bigr),\\
&{\rm(c)}& \limsup\limits_{x\to\infty} \sup\limits_{n \geqslant1}\frac{1}{n\overline{F}_{\xi_1}(x)}\sum\limits_{i=1}^{n}\overline{F}_{\xi_i}(x)<\infty,\\
&{\rm(d)}& \mathbb{E} \eta^{p+1} < \infty\ \ \mathit{for \ some} \ p > J^{+}_{F_{\xi_1}}.
\end{eqnarray*}
\end{thm}

In this work, we consider the randomly stopped maximum $\xi_{(\eta)}$
and randomly stopped maximum of sums $S_{(\eta)}$ of independent but
not necessarily identically\vadjust{\eject} distributed r.v.s. As was noted before, we
restrict our consideration to the class $\mathcal{C}$ but extend it to
the real-valued r.v.s as in Theorem~\ref{tt3}.

If r.v.s $\{\xi_1,\xi_2,\ldots\}$ are not identically distributed,
then different collections of conditions on r.v.s $\{\xi_1,\xi_2,\ldots
\}$ and $\eta$ imply that $F_{\xi_{(\eta)}}\in\mathcal{C}$ or
$F_{S_{(\eta)}}\in\mathcal{C}$. We suppose that some r.v.s from $\{\xi
_1,\xi_2, \ldots\}$ have distributions belonging to the class $
\mathcal{C}$, and we find conditions for r.v.s $\{\xi_1,\xi_2, \ldots\}
$ and $\eta$ such that the distribution of the randomly stopped maximum
or the randomly stopped maximum of sums remains in the same class. The
results presented and their proofs are closely related to the results
of the papers \cite{ds-2016,dss-2016,kss-2016}.

It is worth noting that the closure properties for d.f.s $F_{S_{(\eta
)}}$ in the case of i.i.d. r.v.s can be derived, for instance, from the
asymptotic formulas obtained in \cite{dfk-2010,kt-2003,nty-2002,wt-2004,zcw-2011}. Unfortunately, in the
case of nonidentically distributed r.v.s, similar asymptotic formulas
do not exist. Therefore we have to use other methods to prove our main results.

The rest of the paper is organized as follows. In Section~\ref{main},
we present our main results together with a few examples of randomly
stopped maximum $\xi_{(\eta)}$ and randomly stopped maximum of sums
$S_{(\eta)}$ with d.f.s having consistently varying tails. Section~\ref{lemma} is a collection of auxiliary lemmas, and the proofs of the main
results are presented in Section~\ref{proof}.

\section{Main results}\label{main}

In this section, we present four statements, two theorems and two
corollaries. Theorem~\ref{th1} and Corollary~\ref{c1} deal with the
belonging of the d.f. $F_{\xi_{(\eta)}}$ to the class~$\mathcal{C}$.

\begin{thm}\label{th1}  Let $\{\xi_1,\xi_2,\dots\}$ be a
sequence of independent real-valued r.v.s, and let $\eta$ be a counting
r.v. independent of $\lbrace\xi_1,\xi_2,\dots\rbrace$. The d.f. of the
randomly stopped maximum $F_{\xi_{(\eta)} }$ belongs to the class
$\mathcal{C}$ if the following conditions hold:
\begin{align*}
&(a)\hspace{0.1cm}F_{\xi_\varkappa} \in\mathcal{C}\ \text{for some }\varkappa \in {\rm supp}(\eta),\\
&(b)\hspace{0.1cm}\text{for each }\ k\neq\varkappa, \ \text{either }\ F_{\xi_k} \in\mathcal{C}\ \text{or }\ \overline{F}_{\xi_k}(x)= o \bigl(\overline{F}_{\xi_\varkappa}(x)\bigr), \\
&(c)\hspace{0.1cm}\limsup_{x\to\infty}\sup_{n\geqslant1} \frac{1}{\varphi(n)\overline{F}_{\xi_\varkappa}(x)}\sum\limits_{k=1}^n\overline{F}_{\xi_k}(x)<\infty,
\end{align*}
where $\{\varphi(n)\}_{n=1}^\infty$ is a positive sequence such that
$\mathbb{E} (\varphi(\eta)\ind_{[1,\infty)}(\eta) )<\infty$.
\end{thm}

If r.v.s $\{\xi_1,\xi_2,\ldots\}$ are identically distributed with
common d.f. $F_{\xi} \in\mathcal{C}$, then conditions (a), (b), and
(c) are satisfied with $\varphi(n)=n$ for $n\in\mathbb{N}$. Hence, the
following statement immediately follows from Theorem~\ref{th1}.

\begin{cor}\label{c1} Let $\{\xi_1,\xi_2,\dots\}$ be a
sequence of i.i.d. real-valued r.v.s with common d.f. $F_\xi\in\mathcal
{C}$, and let $\eta$ be a counting r.v. independent of $\lbrace\xi
_1,\xi_2,\dots\rbrace$. The d.f.\ of the randomly stopped maximum
$F_{\xi_{(\eta)} }$ belongs to the class $\mathcal{C}$ if \,$\mathbb
{E}\eta$\, is finite.
\end{cor}

In Theorem~\ref{th2} and Corollary~\ref{c2}, we present conditions
under which the d.f. of the randomly stopped maximum of sums $S_{(\eta
)}$ has consistently varying tail.

\begingroup
\abovedisplayskip=7.5pt
\belowdisplayskip=7.5pt
\begin{thm}\label{th2}
Let $\{\xi_1,\xi_2,\dots\}$ be independent real-valued r.v.s,
and let $\eta$ be a~counting r.v. independent of $\lbrace\xi_1,\xi
_2,\dots\rbrace$. The d.f. $F_{S_{(\eta)} }$ belongs to the class~
$\mathcal{C}$ if the following conditions hold:\vadjust{\eject}
\begin{align*}
\notag&(a)\hspace{0.1cm}F_{\xi_k} \in\mathcal{C} \text{ for each}\ k \in \mathbb{N},\\
\notag&(b)\hspace{0.1cm}\limsup_{x\to\infty}\sup_{n\geqslant1}\frac{1}{n\overline{F}_{\xi_1}(x)}\sum\limits_{k=1}^{n}\overline{F}_{\xi_k}(x)<\infty,\\
\notag&(c)\hspace{0.1cm}\mathbb{E}\eta^{p+1}<\infty\text{ for some }p>J_{F_{\xi_1}}^+.
\end{align*}
\end{thm}

Similarly to Corollary~\ref{c1}, we can state the following corollary.
We note that in the i.i.d. case condition (b) is obviously satisfied if
$\overline{F}_{\xi_1}(x)>0$ for all $x\in\mathbb{R}$.

\begin{cor}\label{c2}
\textit{Let $\{\xi_1,\xi_2,\dots\}$ be i.i.d. real-valued r.v.s with
common d.f. $F_\xi\in\mathcal{C}$, and let $\eta$ is a counting r.v.
independent of $\lbrace\xi_1,\xi_2,\ldots\rbrace$. The d.f.\  $F_{S_{(\eta
)} }$ belongs to the class $\mathcal{C}$ if $\mathbb{E}\eta^{p+1}<\infty
$ for some $p>J_{F_{\xi}}^+$.}
\end{cor}

Further in this section, we present two examples of r.v.s $\{\xi_1,\xi
_2,\ldots\}$ and $\eta$ showing the applicability of our theorems.
\begin{ex}\label{ex1}
Let $\{\xi_1,\xi_2,\ldots\}$ be independent r.v.s such that:
\begin{itemize}
\item r.v.s $\xi_k$ are exponentially distributed for $k\equiv\,0\,{\rm mod}\,3$, that is,
\[
{\overline{F}}_{\xi_k}(x)=\ind_{(-\infty,0)}(x)+{\rm e}^{-x}
\ind _{[0,\infty)}(x),\quad  k\in\{3,6,9,\ldots\};
\]
\item r.v.s $\xi_k$ are degenerate at zero for all $k\equiv
\, 2\,{\rm mod}\,3$ and $k\equiv\, 1\,{\rm mod}\,3, k\geqslant4$, that~is,
\[
{F}_{\xi_k}(x)=\ind_{[0,\infty)}(x),\quad  k\in\{2,5,8,\ldots\}\cup\{4,7,10,\ldots\};
\]
\item $\xi_1=(1+\mathcal{U})2^{\mathcal{G}}$, where r.v.s
$\mathcal{U}$ and $\mathcal{G}$ are independent, $\mathcal{U}$ is
uniformly distributed on the interval $[0,1]$, and $\mathcal{G}$ is
geometrically distributed with parameter $q\in(0,1)$, that is,
\[
\mathbb{P}(\mathcal{G}=l)=(1-q)q^l,\quad  l\in\{0,1,2,\ldots\}.
\]
In addition, let $\eta$ be a counting r.v. independent of $\{\xi_1,\xi
_2,\ldots\}$.
\end{itemize}

Theorem~\ref{th1} implies that the d.f. of the randomly stopped maximum
$\xi_{(\eta)}$ belongs to the class $\mathcal{C}$ if $1\in{\rm
supp}(\eta)$ and $\mathbb{E}\eta<\infty$ because:
\begin{itemize}
\item $F_{\xi_1}\in\mathcal{C}$, but $F_{\xi_1}\notin\mathcal{R}$ due to considerations in pp. 122--123 of \cite{cai-tang},
\item $\overline{F}_{\xi_k}(x)=o(\overline{F}_{\xi_1}(x))$ for each $k\neq1$,
\item $\sup\limits_{n\geqslant1}\frac{1}{n\overline{F}_{\xi_1}(x)}\sum\limits_{k=1}^n\overline{F}_{\xi_k}(x)\leqslant2$ for sufficiently large $x$.
\end{itemize}
\end{ex}

\begin{ex}\label{ex2} Let $\{\xi_1,\xi_2,\ldots\}$ be
independent r.v.s distributed according to Pareto-type laws, namely,
\begin{eqnarray*}
{\overline{F}}_{\xi_k}(x)=\ind_{(-\infty,0)}(x)+\frac{1}{(1+x)^3}\ind_{[0,\infty)}(x) \quad \text{if}\ k\in\{1,3,5,\ldots\},\\
{\overline {F}}_{\xi_k}(x)=\ind_{(-\infty,-1)}(x)+\frac{1}{(2+x)^3}\ind_{[-1,\infty)}(x) \quad \text{if}\ k\in\{2,4,6,\ldots\}.
\end{eqnarray*}
Further, let $\eta$ be a counting r.v. independent of $\{\xi_1,\xi_2,\ldots\}$ that has the zeta distribution with parameter 6, that is,
\[
\mathbb{P}(\eta=m)=\frac{1}{\zeta(6)}\frac{1}{(m+1)^6},\quad  m\in\{ 0,1,2,\ldots\},
\]
where $\zeta$ denotes the Riemann zeta function.

\medskip

Theorem~\ref{th2} implies that the d.f. of the randomly stopped maximum
of sums $S_{(\eta)}$ belongs to the class $\mathcal{C}$ because:
\begin{itemize}
\item $F_{\xi_k}\in\mathcal{R}\subset\mathcal{C}$ for each $k\in\mathbb{N}$;
\item $\sup\limits_{n\geqslant1}\frac{1}{n\overline{F}_{\xi_1}(x)}\sum\limits_{k=1}^n\overline{F}_{\xi_k}(x)=1$ for positive $x$;
\item $J^+_{F_{\xi_1}}=3$ and $\mathbb{E}\eta^{4.5}<\infty$.
\end{itemize}
\end{ex}
\endgroup

\section{Auxiliary lemmas}
\label{lemma}

In this section, we give auxiliary lemmas, which we use in the proof of
Theorem~\ref{th2}. The first lemma was proved in \cite
[Thm.~2.1]{nty-2002}; a more general case can be found in \cite
[Thm.~2.1]{chen2}. We recall only that $\mathcal{C}\subset\mathcal{L}$.
Hence, the statement of the next lemma holds for d.f.s with
consistently varying tails.

\begin{lemma}\label{l1}
Let $\{X_1,X_2,\ldots,X_n\}$ be independent real-valued r.v.s such
that $F_{X_{k}}\in\mathcal{L}$ for all $k\in\mathbb{N}$. Then
\begin{align}
\notag\mathbb{P} \Biggl( \max_{1\leqslant k\leqslant n} \sum\limits_{i=1}^{k}X_i>x \Biggr) \underset{x\to\infty}\sim\mathbb{P} \Biggl(\sum\limits_{i=1}^{n}X_i>x\Biggr).
\end{align}
\end{lemma}

The second auxiliary lemma was proved in \cite[Lemma 3]{kss-2016}. It
describes the situation where the d.f. of sums of independent r.v.s
belongs to the class $\mathcal{C}$.

\begin{lemma}\label{l2}
Let $\{X_1,X_2,\ldots,X_n\}$ be independent real-valued r.v.s. The
d.f. of the sum $\varSigma_n:=X_1+X_2+\cdots+X_n$ belongs to the class
$\mathcal{C}$ if the following two conditions are satisfied:
\begin{align}
\notag & (a)\hspace{0.1cm}F_{X_1} \in\mathcal{C}, \\
\notag & (b)\hspace{0.1cm}\text{for each } k \in\{ 2,\dots,n\},\ \text {either}\ F_{X_k} \in\mathcal{C}\ \text{or}\ \overline{F}_{X_k}(x)= o \bigl(\overline{F}_{X_1}(x)\bigr). \end{align}
\end{lemma}

The following statement was proved in \cite[Lemma 3.2]{ds-2016}. It
gives an upper estimate for the tail of the sum of r.v.s from the class
$\mathcal{D}$.

\begin{lemma}\label{l3}
Let $\{X_1,X_2,\dots\}$ be nonnegative independent r.v.s,
$F_{X_\nu}\in\mathcal{D}$ for some $\nu\geqslant1$ such that
\begin{align}
\notag\limsup_{x\to\infty}\sup_{n\geqslant\nu}\frac{1}{n\overline{F}_{X_\nu}(x)}\sum\limits_{i=1}^{n}\overline{F}_{X_i}(x)<\infty.
\end{align}
Then, for each $p>J_{F_{X_\nu}}^+$, there exists a positive constant
$c_1$ such that
\begin{align}
\notag\overline{F}_{\varSigma_n}(x)\leqslant c_1 n^{p+1}\overline {F}_{X_{\nu}}(x),
\end{align}
for all $n\geqslant\nu$ and $x\geqslant0$, where $F_{\varSigma_n}$ is the
d.f. of the sum $\varSigma_n=X_1+\cdots+X_n$.
\end{lemma}

The last useful lemma is an obvious conclusion from Theorem 3.1 in \cite{chen}.

\begin{lemma}\label{l4}
Let $\{X_1,X_2, \ldots X_n\}$ be independent real-valued r.v.s.
If ${F}_{X_{k}} \in\mathcal{C}$ for each $k\in\{1,2,\ldots,n \}$, then
\[
\mathbb{P} \Biggl(\,\sum\limits_{i=1}^{n}X_i>x\Biggr)\sim\sum\limits_{i=1}^{n}\overline{F}_{X_i}(x).
\]
\end{lemma}

\section{Proofs of the main results}\label{proof}

\begin{proof}[Proof of Theorem~\ref{th1}]
It suffices to prove that
\begin{equation}
\label{eq:tikslas1} \limsup_{y\uparrow1}\limsup_{x\to\infty}
\frac{\overline{F}_{\xi_{(\eta
)} }(xy)}{\overline{F}_{\xi_{(\eta)}}(x)}\leqslant1.
\end{equation}
For all $x>0$ and $K\in\mathbb{N}$, we have
\begin{align*}
\overline{F}_{\xi_{(\eta)}}(x)&=\sum\limits_{n=1}^{\infty}\overline {F}_{\xi_{(n)}}(x)\mathbb{P}(\eta=n)\\
&= \Biggl(\sum\limits_{n=1}^{K}+\sum\limits_{n=K+1}^{\infty}\Biggr)\mathbb {P} \Biggl(\bigcup\limits_{k=1}^{n}\{\xi_k>x\} \Biggr)\mathbb{P}(\eta=n).
\end{align*}
Therefore,
\begin{align}
\label{eq:suma} \notag\frac{\overline{F}_{\xi_{(\eta)} }(xy)}{\overline{F}_{\xi_{(\eta)}}(x)}&=\frac{\sum_{n=1}^{K}\mathbb{P}\big(\bigcup_{k=1}^{n}\{\xi_k>xy\}\big)\mathbb{P}(\eta=n)}{\mathbb{P}(\xi_{(\eta)}>x)}\\
\notag&\quad +\frac{\sum_{n=K+1}^{\infty}\mathbb{P} \big(\bigcup_{k=1}^{n}\{\xi_k>xy\}\big)\mathbb{P}(\eta=n)}{\mathbb{P}(\xi_{(\eta)}>x)}\\
&=:\mathcal{J}_1+\mathcal{J}_2,
\end{align}
for all $x>0$ and $y\in(0,1)$.

Denote
\[
\mathcal{K}:= \bigl\{k\in\mathbb{N}:F_{\xi_k}\notin\mathcal{C}\ \text
{and}\ \overline{F}_{\xi_k}(x)=o \bigl(\overline{F}_{\xi_\varkappa}(x)
\bigr) \bigr\}.
\]

If $x>0$, $1/2\leqslant y<1$, and $K\geqslant\varkappa$, then
\begin{equation}
\label{0-0} \mathcal{J}_1\leqslant\frac{1}{\overline{F}_{\xi_{(\eta)}}(x)}\sum
\limits
_{n=1}^{K}\sum\limits
_{k=1\atop k\notin\mathcal{K}}^{n}
\overline {F}_{\xi_{k}}(xy)\mathbb{P}(\eta=n)+ \sum\limits
_{n=1}^{K}
\sum\limits
_{k=1\atop k\in\mathcal{K}}^{n}\frac
{\overline{F}_{\xi_{k}}(x/2)}{\overline{F}_{\xi_{(\eta)}}(x)}\mathbb {P}(
\eta=n).
\end{equation}

If $k\in\mathcal{K}$, then
\begin{equation}
\label{0-1} \overline{F}_{\xi_k}(x/2)=o \bigl(\,\overline{F}_{\xi_{(\eta)}}(x)
\bigr),
\end{equation}
because
\begin{equation*}
\frac{\overline{F}_{\xi_k}(x/2)}{\overline{F}_{\xi_{(\eta
)}}(x)}\leqslant \frac{\overline{F}_{\xi_k}(x/2)}{\overline{F}_{\xi_{(\varkappa
)}}(x)\mathbb{P}(\eta=\varkappa)} \leqslant\frac{\overline{F}_{\xi_k}(x/2)}{
\overline{F}_{\xi_\varkappa}(x/2)P(\eta=\varkappa)}
\frac{\overline
{F}_{\xi_\varkappa}(x/2)}{\overline{F}_{\xi_\varkappa}(x)},
\end{equation*}
and $F_{\xi_\kappa}\in\mathcal{C}\subset\mathcal{D}$.

The obtained asymptotic relation \eqref{0-1} and
estimate \eqref{0-0} imply
\begin{align}\label{0-2}
\limsup\limits _{x\rightarrow\infty}\mathcal{J}_1&\leqslant\limsup \limits_{x\rightarrow\infty} \frac{1}{\overline{F}_{\xi_{(\eta)}}(x)}\sum\limits _{n=1}^{K}\sum \limits _{k=1\atop k\notin\mathcal{K}}^{n}\overline {F}_{\xi_{k}}(\mathit{xy})\mathbb{P}( \eta=n) \nonumber                                                                                                   \\
&\leqslant \limsup\limits _{x\rightarrow\infty} \Biggl(\max\limits _{1\leqslant k\leqslant K\atop k\notin\mathcal{K}} \biggl\{\frac{\overline {F}_{\xi_{k}}(\mathit{xy})}{\overline{F}_{\xi_{k}}(x)} \biggr\} \frac{1}{\overline {F}_{\xi_{(\eta)}}(x)}\sum\limits _{n=1}^{K}\sum \limits _{k=1\atop k\notin\mathcal{K}}^{n}\overline{F}_{\xi_{k}}(x)\mathbb{P}( \eta=n) \Biggr) \nonumber \\
&\leqslant \max\limits _{1\leqslant k\leqslant K\atop k\notin\mathcal{K}} \biggl\{\limsup\limits _{x\rightarrow\infty } \frac{\overline{F}_{\xi_{k}}(\mathit{xy})}{\overline{F}_{\xi_{k}}(x)} \biggr\} \nonumber                                                                                                                                                  \\
&\quad \times \limsup\limits _{x\rightarrow\infty} \Biggl\{\frac{1}{\overline{F}_{\xi _{(\eta)}}(x)}\sum \limits _{n=1}^{K}\sum\limits _{k=1}^{n} \overline {F}_{\xi_{k}}(x)\mathbb{P}(\eta=n) \Biggr\}.
 \end{align}

For each $1\leqslant n\leqslant K$, we have
\[
\overline{F}_{\xi_{(n)}}(x)=\mathbb{P} \Biggl(\bigcup
\limits_{k=1}^{n}
\{ \xi_k>x\} \Biggr) \geqslant\sum\limits
_{k=1}^{n}
\overline{F}_{\xi_k}(x) \Biggl(1-\sum\limits
_{k=1}^{K}
\overline{F}_{\xi_k}(x) \Biggr),
\]
due to the Bonferroni inequality. This implies that, for an arbitrary
$\varepsilon>0$,
\[
\sum\limits
_{k=1}^{n}\overline{F}_{\xi_k}(x)
\leqslant(1+\varepsilon )\overline{F}_{\xi_{(n)}}(x),
\]
if $x$ is sufficiently large.

Substituting the last estimate into inequality \eqref
{0-2}, we get
\begin{align*}
\limsup\limits_{y\uparrow1}\limsup\limits _{x\rightarrow\infty}\mathcal {J}_1 & \leqslant (1+\varepsilon) \max\limits _{1\leqslant k\leqslant K\atop k\notin\mathcal{K}} \biggl\{\limsup\limits _{y\uparrow1}\limsup\limits _{x\rightarrow\infty } \frac{\overline{F}_{\xi_{k}}(\mathit{xy})}{\overline{F}_{\xi_{k}}(x)} \biggr\} \\
& = 1+\varepsilon,
\end{align*}
because $F_{\xi_k}\in\mathcal{C}$ for each $k\notin\mathcal{K}$.

Since $\varepsilon>0$ is arbitrary, the last estimate
implies that
\begin{align}
\label{eq:pirmasJ} \limsup_{y\uparrow1}\limsup_{x\rightarrow\infty}
\mathcal{J}_1\leqslant1.
\end{align}

Since
\[
\mathbb{P}(\xi_{(\eta)}>x)\geqslant\mathbb{P}(\xi_{(\varkappa
)}>x)
\mathbb{P}(\eta=\varkappa)\geqslant\mathbb{P}(\xi_{\varkappa
}>x)\mathbb{P}(
\eta=\varkappa),
\]
we obtain
\begin{align*}
\mathcal{J}_2\leqslant\frac{\sum_{n=K+1}^{\infty}\sum
_{k=1}^{n}\overline{F}_{\xi_k}(\mathit{xy})\mathbb{P}(\eta=n)}{\overline{F}_{\xi
_\varkappa}(x)
\mathbb{P}(\eta=\varkappa)},
\end{align*}
for all $x>0$ and $y\in(0,1)$.

Condition (c) of the theorem implies that, for some constant $c_2>0$,
\begin{align*}
\sum\limits
_{k=1}^{n}\overline{F}_{\xi_k}(x)
\leqslant c_2\varphi(n) \overline{F}_{\xi_\varkappa}(x),
\end{align*}
for all sufficiently large $x$ and all $n\geqslant1$.

Consequently,
\begin{align}\label{eq:antrasJ}
& \limsup_{y\uparrow1}\limsup_{x\to\infty}\mathcal{J}_2 \nonumber                                                                                                                                                                       \\
&\quad \leqslant \limsup_{y\uparrow1}\limsup_{x\to\infty} \frac{c_2 \overline {F}_{\xi_\varkappa}(\mathit{xy})\sum_{n=K+1}^{\infty}\varphi(n)\mathbb {P}(\eta=n)}{\overline{F}_{\xi_\varkappa}(x)\mathbb{P}(\eta=\varkappa )} \nonumber                  \\
&\quad = \frac{c_2}{\mathbb{P}(\eta=\varkappa)} \biggl(\limsup_{y\uparrow 1}\limsup _{x\to\infty}\frac{\overline{F}_{\xi_\varkappa}(\mathit{xy})}{\overline {F}_{\xi_\varkappa}(x)} \biggr)\sum\limits _{n=K+1}^{\infty} \varphi (n)\mathbb{P}(\eta=n).
\end{align}

Relations \eqref{eq:suma}, \eqref{eq:pirmasJ}, and \eqref{eq:antrasJ}
imply that
\begin{align*}
\limsup_{y\uparrow1}\limsup_{x\to\infty}\frac{\overline{F}_{\xi_{(\eta
)} }(\mathit{xy})}{\overline{F}_{\xi_{(\eta)}}(x)}
\leqslant1+\frac{c_2}{\mathbb
{P}(\eta=\varkappa)}\mathbb{E} \bigl(\varphi(\eta)\ind_{[K+1,\infty)}(
\eta ) \bigr),
\end{align*}
for arbitrary $K\geqslant\varkappa$.

The desired inequality \eqref{eq:tikslas1} now follows from the last
estimate because\break $\mathbb{E} (\varphi(\eta)\ind_{[1,\infty)}(\eta
) )$ is finite due to the conditions of the theorem. Theorem~\ref
{th1} is proved.
\end{proof}

\begin{proof}[Proof of Theorem~\ref{th2}]
Similarly as in the proof
of Theorem~\ref{th1}, it suffices to show that
\begin{align}
\label{eq:tikslas} \limsup_{y\uparrow1}\limsup_{x\to\infty}
\frac{\overline{F}_{S_{(\eta)}
}(\mathit{xy})}{\overline{F}_{S_{(\eta)}}(x)}\leqslant1.
\end{align}

If $K\in\mathbb{N}$ and $x>0$, then
\begin{align}
\notag\mathbb{P}(S_{(\eta)}>x)= \Biggl(\sum\limits
_{n=1}^{K}+
\sum\limits
_{n=K+1}^{\infty} \Biggr)\mathbb{P}(S_{(n)}>x)
\mathbb{P}(\eta=n).
\end{align}

Therefore,
\begin{align}
\label{eq:J1+J2} \notag\frac{\mathbb{P}(S_{(\eta)}>\mathit{xy})}{\mathbb{P}(S_{(\eta)}>x)}&=\frac
{\sum_{n=1}^{K}\mathbb{P}(S_{(n)}>\mathit{xy})\mathbb{P}(\eta=n)}{\mathbb
{P}(S_{(\eta)}>x)}
\\
\notag&\quad +\frac{\sum_{n=K+1}^{\infty}\mathbb{P}(S_{(n)}>\mathit{xy})\mathbb
{P}(\eta=n)}{\mathbb{P}(S_{(\eta)}>x)}
\\
&=:\mathcal{I}_{1}+\mathcal{I}_{2},
\end{align}
for all $x>0$ and $y\in(0,1)$.

The r.v. $\eta$ is not degenerate at zero. Therefore, there exists
$a\in\mathbb{N}$ such that $\mathbb{P}(\eta=a)>0$. If $K\geqslant a$, then
\begin{align}
\notag\mathcal{I}_{1}\leqslant\frac{\sum_{\substack{n=1}}^{K}\mathbb{P}(S_{(n)}>\mathit{xy})\mathbb{P}(\eta=n)}{\sum_{\substack{n=1}}^{K}\mathbb
{P}(S_{(n)}>x)\mathbb{P}(\eta=n)}\leqslant\max\limits
_{\substack
{1\leqslant n\leqslant K\\
n\in{\rm supp}(\eta)}}
\frac{\mathbb
{P}(S_{(n)}>\mathit{xy})}{\mathbb{P}(S_{(n)}>x)},
\end{align}
due to the inequality
\begin{align*}
\frac{a_1+a_2+\cdots+a_m}{b_1+b_2+\cdots+b_m}\leqslant\max \biggl\{\frac
{a_1}{b_1},\frac{a_2}{b_2},
\dots,\frac{a_m}{b_m} \biggr\},
\end{align*}
provided for all $a_i\geqslant0$, $b_i>0$, $i\in\{1,2,\dots,m\}$, and $m\in
\mathbb{N}$.

Since $\mathcal{C}\subset\mathcal{L}$, using Lemma~\ref{l1}, we obtain
\begin{align*}
& \limsup_{y\uparrow1}\limsup_{x\to\infty}\mathcal{I}_1 \\
&\quad \leqslant\limsup_{y\uparrow1}\limsup_{x\to\infty}\max \limits _{\substack{1\leqslant n\leqslant K \\
n\in{\rm supp}(\eta)}}\frac {\overline{F}_{S_{(n)} }(\mathit{xy})}{\overline{F}_{S_{(n)}}(x)} \\
\notag & \quad =\max\limits _{\substack{1\leqslant n\leqslant K \\
n\in{\rm supp}(\eta)}}\limsup_{y\uparrow1}\limsup _{x\to\infty}\frac{\overline {F}_{S_{(n)} }(\mathit{xy})}{\overline{F}_{S_{n}}(\mathit{xy})}\frac{\overline{F}_{S_{n} }(\mathit{xy})}{\overline{F}_{S_{n}}(x)}\frac{\overline{F}_{S_{n} }(x)}{\overline {F}_{S_{(n)} }(x)} & \\
\notag & \quad \leqslant\max\limits _{\substack{1\leqslant n\leqslant K}}\limsup_{y\uparrow1}\limsup _{x\to\infty}\frac{\overline{F}_{S_{n} }(\mathit{xy})}{\overline{F}_{S_{n}}(x)}.
\end{align*}

According to Lemma~\ref{l2}, the d.f. $F_{S_n}$ belongs to the class
$\mathcal{C}$ for each fixed~$n$. Therefore,
\begin{align}
\label{eq:J1} \limsup_{y\uparrow1}\limsup_{x\to\infty}
\mathcal{I}_1\leqslant1.
\end{align}

Let us now consider the term $\mathcal{I}_2$ from expression \eqref
{eq:J1+J2}. If $x>0$, then the numerator of $\mathcal{I}_2$ can be
estimated as follows:
\begin{align*}
\notag&\sum\limits_{n=K+1}^{\infty}\mathbb{P}(S_{(n)}>\mathit{xy})\mathbb{P}(\eta=n)\\
&\quad =\sum\limits_{n=K+1}^{\infty}\mathbb {P}\bigl(\max\{S_1,\dots,S_n\}>\mathit{xy}\bigr)\mathbb{P}(\eta=n)\\
\notag&\quad \leqslant\sum\limits_{n=K+1}^{\infty}\mathbb{P}\bigl(\max\bigl\{ S_1^{(+)},\dots,S_n^{(+)}\bigr\}>\mathit{xy}\bigr)\mathbb{P}(\eta=n)\\
\notag&\quad =\sum\limits_{n=K+1}^{\infty}\mathbb{P}\bigl(S_n^{(+)}>\mathit{xy}\bigr)\mathbb {P}(\eta=n),
\end{align*}
where $S_k^{(+)}:=\xi_1^{+}+\cdots+\xi_k^{+}$ and $\xi^+:=\xi\ind
_{[0,\infty)}(\xi)$.

We can apply Lemma~\ref{l3} for the last sum because of condition (b)
of the theorem and the fact that $F_{\xi_1}\in\mathcal{C}$ $\Rightarrow
$ $F_{\xi_1^+}\in\mathcal{D}$. Using this lemma, we get
\begin{equation}
\label{ai1} \sum\limits
_{n=K+1}^{\infty}\mathbb{P}(S_{(n)}>\mathit{xy})
\mathbb{P}(\eta=n) \leqslant c_3\overline{F}_{\xi_1}(\mathit{xy})\sum
\limits
_{n=K+1}^{\infty
}n^{p+1}\mathbb{P}(\eta=n),
\end{equation}
for some positive constant $c_3$ and for all $x>0$ and $y\in(0,1)$.

On the other hand, for the denominator of $\mathcal{I}_2$, we have
\begin{align*}
\mathbb{P}(S_{(\eta)}>x)&=\sum\limits
_{n=1}^{\infty}\mathbb
{P}(S_{(n)}>x)\mathbb{P}(\eta=n)
\\
&\geqslant\mathbb{P}\bigl(\max\{S_1,\dots,S_a\}>x\bigr)
\mathbb{P}(\eta=a)
\\
&\geqslant\mathbb{P}(S_a>x)\mathbb{P}(\eta=a).
\end{align*}

According to the conditions of the theorem, the d.f. $F_{\xi_k}$
belongs to the class~$\mathcal{C}$ for each fixed index $k\in\mathbb
{N}$. Hence, using Lemma~\ref{l4}, we get
\begin{align*}
\liminf_{x\to\infty}\frac{\mathbb{P}(S_a>x)}{\overline{F}_{\xi_1}(x)} &\geqslant\liminf
_{x\to\infty}\frac{\raisebox{0.7ex}{$\displaystyle
\mathbb{P}(S_a>x)$}}{\sum_{i=1}^a\overline{F}_{\xi_i}(x)}\liminf_{x\to\infty}
\frac{\sum_{i=1}^a\overline{F}_{\xi_i}(x)}{\raisebox
{-0.9ex}{$\displaystyle\overline{F}_{\xi_1}(x)$}}
\\
&=\liminf_{x\to\infty}\frac{\sum_{i=1}^a\overline{F}_{\xi
_i}(x)}{\raisebox{-0.8ex}{$\displaystyle\overline{F}_{\xi_1}(x)$}} \geqslant\liminf
_{x\to\infty}\frac{\overline{F}_{\xi_1}(x)}{\overline
{F}_{\xi_1}(x)} =1.
\end{align*}

The last two estimates imply that
\begin{align}
\label{ai2} \mathbb{P}(S_{(\eta)}>x)\geqslant\frac{1}{2}
\overline{F}_{\xi
_1}(x)\mathbb{P}(\eta=a),
\end{align}
for sufficiently large $x$.

Therefore, by inequalities \eqref{ai1} and \eqref{ai2} we have
\begin{align}\label{eq:J2}
\notag&\limsup_{y\uparrow1}\limsup
_{x\to\infty}\mathcal {I}_2
\\
&\quad \leqslant \frac{2c_3}{\mathbb{P}(\eta=a)} \biggl(\limsup_{y\uparrow
1}\limsup
_{x\to\infty}\frac{\overline{F}_{\xi_1}(\mathit{xy})}{\overline{F}_{\xi
_1}(x)} \biggr)\sum\limits
_{n=K+1}^{\infty}n^{p+1}
\mathbb{P}(\eta=n).
\end{align}

Finally, substituting estimates \eqref{eq:J1} and \eqref{eq:J2} into
\eqref{eq:J1+J2}, we obtain that
\begin{align*}
\limsup_{y\uparrow1}\limsup_{x\to\infty}\frac{\mathbb{P}(S_{(\eta
)}>\mathit{xy})}{\mathbb{P}(S_{(\eta)}>x)}
\leqslant1+\frac{2c_3}{\mathbb{P}(\eta
=a)}\mathbb{E} \bigl(\eta^{p+1}
\ind_{[K+1,\infty)}(\eta) \bigr),
\end{align*}
for an arbitrary $K\geqslant a$.

This inequality implies the desired relation \eqref{eq:tikslas} due to
condition (c) of the theorem. Theorem~\ref{th2} is proved.
\end{proof}

\section*{Acknowledgments} We would like to thank the anonymous
referees for the detailed and helpful comments on the first version of
the manuscript and especially for the proposed new way of proving
Theorem 4.



\end{document}